\documentclass[final]{siamltex}
\usepackage{amsfonts}
\usepackage{amssymb}
\usepackage{amsmath}
\usepackage{epsfig}
\def\a{a}
\def\b{\vec{\mathbf{\beta}}}

\newtheorem{remark}[theorem]{Remark}
\title{Preconditioned HSS Method for Finite Element Approximations
of Convection-Diffusion Equations\thanks{July, 23, 2008 -
The work
of the first author was partially supported by MIUR,
grant number 2006013187 and the work of the second author was partially
supported by MIUR, grant number 2006017542.}}
\author{Alessandro Russo\thanks{Dipartimento di Matematica e Applicazioni,
Universit\`a di Milano Bicocca, via Cozzi 53, 20125 Milano, Italy
({\tt alessandro.russo@unimib.it}).}
        \and Cristina Tablino Possio\thanks{Dipartimento di Matematica e Applicazioni,
Universit\`a di Milano Bicocca, via Cozzi 53, 20125 Milano, Italy
({\tt cristina.tablinopossio@unimib.it}).}}
\begin{document}
\maketitle
\begin{abstract}
A two-step preconditioned iterative method based on the Hermitian/Skew-Hermitian splitting
is applied to the solution of nonsymmetric linear systems arising from the Finite Element approximation
of convection-diffusion equations.
The theoretical spectral analysis focuses on the case of matrix sequences related to
FE approximations on uniform structured meshes, by referring to spectral tools
derived from Toeplitz theory. In such a setting, if the problem is coercive, and the diffusive and convective coefficients are regular enough, then the proposed preconditioned matrix sequence shows a strong clustering at unity, i.e.,
a superlinear preconditioning sequence is obtained. Under the same assumptions, the optimality of the PHSS method is
proved and some numerical experiments confirm the theoretical results.
Tests on unstructured meshes are also presented, showing the some convergence behavior.
%
\end{abstract}
\pagestyle{myheadings} \thispagestyle{plain}
\markboth{A. RUSSO AND C. TABLINO POSSIO}{PHSS method for Finite Element
approximations of the convection-diffusion equation}
%
\section{Introduction} \label{sez:introduction}
The paper deals with the numerical solution of linear systems
arising from the Finite Element approximation of the
elliptic convection-diffusion problem
\begin{equation} \label{eq:modello}
\left \{
\begin{array}{l}
%
%
\mathrm{div} \left(-\a(\mathbf{x}) \nabla u +\b(\mathbf{x}) u \right)
=f, \quad \mathbf{x}\in \Omega,\\
u_{|\partial \Omega}=0.
\end{array}
\right.
\end{equation}
%
We apply the two-step iterative method based on the Hermitian/Skew-Hermitian splitting (HSS)
of the coefficient matrix proposed in \cite{BGN-SIMAX-2003} for the solution of nonsymmetric
linear systems whose real part is coercive. The aim is to study the preconditioning effectiveness
when the Preconditioned HSS (PHSS) method is applied, as proposed in \cite{BGST-NM-2005}.
According to the PHSS convergence properties, the preconditioning strategy for the
matrix sequence $\{A_n(\a,\b)\}$ can be tuned only with respect to the
stiffness matrices. 
This allows to adopt the same preconditioning proposal
just analyzed in the case of  Finite Difference (FD) and Finite Element (FE) approximations of the diffusion problem in
\cite{S-NM-1999,ST-LAA-1999,ST-ETNA-2000,ST-ETNA-2003,ST-SIMAX-2003,ST-NA-2001,BS-SINUM-2007},
and, more recently, in the case of FD approximations of (\ref{eq:modello}).\newline
More precisely, we consider  the preconditioning matrix sequence $\{P_n(\a)\}$  defined as
$P_n(\a)=D_n^{{1}/{2}}(\a) A_n(1,0) D_n^{{1}/{2}}(\a)$, where $D_n(\a)=\mathrm{diag}(A_n(\a,0))\mathrm{diag}\!^{-1}(A_n(1,0))$,
i.e., the suitable scaled  main diagonal of $A_n(\a,0)$.
In such a way, the solution of the linear system with coefficient matrix $A_n(\a,\b)$ is reduced to computations
involving diagonals and the matrix $A_n(1,0)$.
In the case of uniform structured meshes the latter task can be efficiently performed by means of fast Poisson solvers,
among which we can list those based on the cyclic reduction idea (see e.g.
\cite{BDGG-SINUM-1971,D-SIAMRev-1970,S-SIAMRev-1977}) and several specialized multigrid methods
(see e.g. \cite{H-Springer-1985,S-NM-2002}). \newline
Our theoretical analysis focuses on the case of matrix sequences $\{A_n(\a,\b)\}$ related to
FE approximations on uniform structured meshes, since the powerful spectral tools
derived from Toeplitz theory \cite{BG-LAA-1998,BS-Springer-1998,S-LAA-1998,S-LAA-1998-2} greatly facilitate
the required spectral analysis. In such a setting, under proper assumptions on $\a(\mathbf{x})$
and $\b(\mathbf{x})$, we prove the optimality of the PHSS method.
In our terminology (see \cite{AN-1993}) this means that the PHSS iterations number for reaching
the solution within a fixed accuracy can be bounded from above by a constant
independent of the dimension $n=n(h)$. \newline
%
%
Nevertheless, the numerical experiments have been performed also in the case of unstructured
meshes, with negligible differences in the PHSS method performances.
In such cases, by coupling the PHSS method with standard  Preconditioned Conjugate Gradient (PCG) and  Preconditioned
Generalized Minimal Residual (PGMRES) method, our proposal makes only use of
matrix vector products (for sparse or even diagonal matrices) and of a solver for the related diffusion
equation with constant coefficient. \newline
The underlying idea is that whenever $P_n(\a)$ results to be
an approximate factorization of $A_n(\a,\b)$, the computational bottleneck lies in the
solution of linear systems with coefficient matrix $A_n(1,0)$. Thus, the main effort in devising efficient
algorithms must be devoted to this simpler problem. Moreover, to some extent, the approximation schemes
in the discretization should take into account this key step to give rise to a linear algebra problem
less difficult to cope with.
\par
The outline of the paper is as follows. In section \ref{sez:fem} we report a brief description of
the FE approximation of the convection-diffusion equation, while in section \ref{sez:PHSS} we
summarize the definition and the convergence properties of the Hermitian/Skew-Hermitian (HSS) method
and of its preconditioned formulation (PHSS method). Section \ref{sez:preconditioning} analyzes the
spectral properties of the matrix sequences arising from FE approximations of the considered
convection-diffusion problem and reports the preconditioner definition. Section \ref{sez:clustering} is
devoted to the theoretical analysis of the  preconditioned matrix sequence spectral properties
in the case of structured uniform meshes. In section \ref{sez:numerical_tests} several numerical experiments
illustrate  the claimed convergence properties and their extension in the case of other structured
and unstructured meshes.
Lastly,  section \ref{sez:conclusions} deals with complexity issues and perspectives.
\section{Finite Element approximation} \label{sez:fem}
Problem (\ref{eq:modello}) can be stated in variational form as follows:
\begin{equation} \label{eq:formulazione_variazionale}
\left \{
\begin{array}{l}
\textrm{find $u \in H_0^1(\Omega)$ such that} \\
\int_\Omega \left ( \a \nabla u \cdot \nabla \varphi -\b \cdot \nabla
\varphi \ u  
\right )
=\int_\Omega f \varphi  \quad \textrm{for all } \varphi \in  H_0^1(\Omega)
\end{array}
\right.
\end{equation}
where $H_0^1(\Omega)$ is the space of square integrable functions, with ${L}^2$ weak derivatives vanishing on
$\partial \Omega$. 
We assume that $\Omega$ is a polygonal domain and we make the following hypotheses on the coefficients
\begin{equation} \label{eq:ipotesi_coefficienti}
\left \{
\begin{array}{l}
\a \in {\bf C}^2(\overline \Omega),\quad \textrm{ with } \a(\mathbf{x}) \ge a_0 >0, \\
\b \in {\bf C}^1(\overline \Omega),\quad \textrm{ with } \mathrm{div} \b \ge 0 \textrm{ pointwise in } \Omega, \\
f \in {L}^2(\Omega).
\end{array}
\right.
\end{equation}
%
%
The previous assumptions guarantee existence and uniqueness for problem (\ref{eq:formulazione_variazionale}).
\newline
For the sake of simplicity, we restrict ourselves to linear finite element approximation of problem
(\ref{eq:formulazione_variazionale}). To this end, let $\mathcal{T}_h=\{K\}$ be a usual finite element partition
of $\overline \Omega$ into triangles, with $h_K=\mathrm{diam}(K)$ and $h=\max_K{h_K}$.
Let $V_h \subset H^1_0(\Omega)$ be the space of linear finite elements, i.e.
\[
V_h=\{\varphi_h : \overline \Omega \rightarrow \mathbb{R} \ \textrm{ s.t. } \varphi_h \textrm{ is continuous,  }
{\varphi_h}_{|_K} \textrm{ is linear, and } 
{\varphi_h}_{| \partial \Omega} =0
 \}.
\]
The finite element approximation of problem (\ref{eq:formulazione_variazionale}) reads:
\begin{equation} \label{eq:formulazione_variazionale_fe}
\left \{
\begin{array}{l}
\textrm{find $u_h \in V_h$ such that} \\
\int_\Omega \left (\a \nabla u_h \cdot \nabla \varphi_h -\b \cdot \nabla
\varphi_h \ u_h 
\right )
=\int_\Omega f \varphi_h \quad \textrm{for all } \varphi_h \in  V_h.
\end{array}
\right.
\end{equation}
For each internal node $i$ of the mesh $\mathcal{T}_h$, let $\varphi_i \in V_h$ be such that $\varphi_i(\textrm{node }i)=1$, and
$\varphi_i(\textrm{node }j)=0$ if $i\ne j$. Then, the collection of all $\varphi_i$'s is a base for $V_h$. We will denote
by $n(h)$ the number of the internal nodes of $\mathcal{T}_h$, which corresponds to the dimension of $V_h$. Then, we write $u_h$ as
\[
u_h=\sum_{i=1}^{n(h)}u_i \varphi_i
\]
and the variational equation (\ref{eq:formulazione_variazionale_fe}) becomes an algebraic linear system:
\begin{equation} \label{eq:modello_discreto}
\sum_{j=1}^{n(h)}\left (\int_\Omega \a \nabla \varphi_j \cdot \nabla
\varphi_i -\nabla \varphi_i \cdot \b \ \varphi_j 
\right ) u_j =\int_\Omega f \varphi_i, \quad i=1,\ldots, n(h).
\end{equation}
The aim of this paper is to study the effectiveness of the proposed Preconditioned HSS method applied to
the quoted nonsymmetric linear systems (\ref{eq:modello_discreto}),
both from the theoretical and numerical point of view.
%
\section{Preconditioned HSS method} \label{sez:PHSS}
In this section we briefly summarize the definition and the relevant properties of the
Hermitian/Skew-Hermitian (HSS) method formulated in \cite{BGN-SIMAX-2003} and of its extension in the case
of preconditioning as proposed in \cite{BGST-NM-2005}.\newline
The HSS method can be applied whenever we are looking for the solution of a linear system
$A_n \mathbf{x} = \mathbf{b}$ where $A_n\in \mathbb{C}^{n\times n}$ is a non singular matrix with a
positive definite real part  and $\mathbf{x}, \mathbf{b}$ belong to  $\mathbb{C}^{n}$.
Several applications in scientific computing lead to such kind of linear problems and, typically,
the matrix $A_n$ is also large and sparse, as in the case of FD or FE approximations of (\ref{eq:modello}).\newline
More in detail, the HSS method refers to the unique Hermitian/Skew-Hermitian splitting of the matrix $A_n$ as
\begin{equation}\label{eq:HSS_splitting}
A_n = \mathrm{Re}(A_n) + \mathrm{i}\, \mathrm{Im} (A_n), \quad
\mathrm{i}^2=-1
\end{equation}
where
\[
\mathrm{Re}(A_n)= \frac{A_n+A_n^H}{2}\quad  \mathrm{and} \quad
\mathrm{Im}(A_n)= \frac{A_n-A_n^H}{2\mathrm{i}}
\]
are Hermitian matrices by definition.\newline
In the same spirit of the ADI method \cite{D-NM-1962}, the quoted splitting allows to define
the Hermitian/Skew-Hermitian (HSS) method \cite{BGN-SIMAX-2003}, as follows \begin{equation}\label{eq:HSS}
\left\{ \begin{array}{lcl}
    \left(\alpha I+\mathrm{Re}(A_n)\right)\mathbf{x}^{k+{1\over 2}} & = &
    \left(\alpha I-\mathrm{i}\, \mathrm{Im}(A_n)\right)\mathbf{\mathbf{x}}^{k}+\mathbf{b} \\
    \left(\alpha I+\mathrm{i}\, \mathrm{Im}(A_n)\right)\mathbf{x}^{k+1} & = &
    \left(\alpha I-\mathrm{Re}(A_n)\right)\mathbf{x}^{k+{1\over 2}}+\mathbf{b}
     \end{array}
\right.
\end{equation}
with $\alpha$ positive parameter and $\mathbf{x}^0$ given initial guess. \newline
Beside the quoted formulation as a two-step iteration method, the HSS method can be reinterpreted as a
stationary iterative method whose iteration matrix is given by
\begin{equation}\label{eq:M-HSS}
\widetilde{M}(\alpha)  =  \left(\alpha I+\mathrm{i}\,
\mathrm{Im}(A_n)\right)^{-1}
      \left(\alpha I-\mathrm{Re}(A_n)\right)
\left(\alpha I+\mathrm{Re}(A_n)\right)^{-1}
      \left(\alpha I-\mathrm{i}\, \mathrm{Im}(A_n)\right)
\end{equation}
and whose convergence properties are only related to the spectral radius of the Hermitian matrix
$\left(\alpha I-\mathrm{Re}(A_n)\right) \left(\alpha I+\mathrm{Re}(A_n)\right)^{-1}$,
which is unconditionally bounded by $1$ provided the positivity of $\alpha$ and of $\mathrm{Re}(A_n)$
\cite{BGN-SIMAX-2003}.\newline
%
Indeed,  the rate of convergence can be unsatisfactory for large values of $n$ in the case of PDEs applications,
as for instance (\ref{eq:modello}), so that the preconditioned formulation of the method proposed
in \cite{BGST-NM-2005} is more profitable.\newline
Let $P_n$ be a Hermitian positive definite matrix. The Preconditioned HSS (PHSS) method
can be defined as
\begin{equation} \label{eq:PHSS}
\left\{ \begin{array}{lcl}
    \left(\alpha I+P_n^{-1}\mathrm{Re}(A_n)\right)\mathbf{x}^{k+{1\over 2}} & = &
    \left(\alpha I-P_n^{-1}\mathrm{i}\, \mathrm{Im}(A_n)\right)\mathbf{x}^{k}+P_n^{-1}\mathbf{b} \\
    \left(\alpha I+P_n^{-1}\mathrm{i}\, \mathrm{Im}(A_n)\right)\mathbf{x}^{k+1} & = &
    \left(\alpha I-P_n^{-1}\mathrm{Re}(A_n)\right)\mathbf{x}^{k+{1\over 2}}+P_n^{-1}\mathbf{b}
     \end{array}
\right.
\end{equation}
Notice that the proposed method differs from the HSS method applied to the matrix $P_n^{-1}A_n$
since $P_n^{-1}\mathrm{Re}(A_n)$ and $P_n^{-1}\mathrm{Im}(A_n)$ are not the
Hermitian/Skew-Hermitian splitting of $P_n^{-1}A_n$.\newline
%
%
Let $\lambda(X)$ denote the set of the eigenvalues of a square matrix $X$;
the convergence properties of the PHSS method exactly mimic those of the previous
HSS method, as claimed in the theorem below. \newline
\begin{theorem} \emph{\cite{BGST-NM-2005}} \label{th:main}
Let $A_n\in \mathbb{C}^{n\times n}$ be a matrix with positive definite real part,
$\alpha$ be a positive parameter and let $P_n\in \mathbb{C}^{n\times n}$ be
a Hermitian positive definite matrix.
Then the iteration matrix of the PHSS method is given by
\begin{eqnarray*}
M(\alpha) & = & \left(\alpha I+\mathrm{i}\,
P_n^{-1}\mathrm{Im}(A_n)\right)^{-1}
      \left(\alpha I-P_n^{-1}\mathrm{Re}(A_n)\right) \left(\alpha I+P_n^{-1}\mathrm{Re}(A_n)\right)^{-1} \\
      & &
      \left(\alpha I-\mathrm{i}\, P_n^{-1}\mathrm{Im}(A_n)\right),
\end{eqnarray*}
its spectral radius $\varrho(M(\alpha))$ is bounded by
\[
\sigma(\alpha)=\max_{\lambda_i\in \lambda(P_n^{-1}{\rm Re}(A_n))}
\left|{\alpha-\lambda_i \over \alpha+\lambda_i}\right| <1 \quad
\textrm{for any} \quad \alpha>0,
\]
i.e., the PHSS iteration is unconditionally convergent to the unique solution
of the system $A_n\mathbf{x}=\mathbf{b}$. Moreover, denoting by
$\kappa={\lambda_{\max}(P_n^{-1}\mathrm{Re}(A_n))}/{\lambda_{\min}(P_n^{-1}\mathrm{Re}(A_n))}$
the spectral condition number (namely the Euclidean (spectral) condition number of the
symmetrized matrix), the optimal $\alpha$ value that minimizes the quantity $\sigma(\alpha)$ equals
\[
\alpha^*=\sqrt{\lambda_{\min}(P_n^{-1}\mathrm{Re}(A_n))\lambda_{\max}(P_n^{-1}\mathrm{Re}(A_n))}
\quad \mathrm{and}\quad
\sigma(\alpha^*)={\sqrt{\kappa}-1\over \sqrt{\kappa}+1}.
\]
\end{theorem}
Thus, the unconditional convergence property holds also in the preconditioned formulation
of the method and the convergence properties are related to the spectral radius of
\begin{equation*}
\left(\alpha I-P_n^{-1/2}\mathrm{Re}(A_n)P_n^{-1/2}\right)
   \left(\alpha I+P_n^{-1/2}\mathrm{Re}(A_n)P_n^{-1/2}\right)^{-1},
\end{equation*}
where the optimal parameter $\alpha$ is the square root of the product
of the extreme eigenvalues of $P_n^{-1}\mathrm{Re}(A_n)$.\newline
It is worth stressing that the PHSS method in (\ref{eq:PHSS}) can also be interpreted as
the original iteration (\ref{eq:HSS}) where the identity matrix is replaced by the
preconditioner $P_n$, i.e.,
\begin{equation} \label{eq:PHSS-P}
\left\{ \begin{array}{lcl}
    \left(\alpha P_n+\mathrm{Re}(A_n)\right)\mathbf{x}^{k+{1\over 2}} & = &
    \left(\alpha P_n-\mathrm{i}\ \mathrm{Im}(A_n)\right)\mathbf{x}^{k} +\mathbf{b} \\
    \left(\alpha P_n+\mathrm{i}\ \mathrm{Im}(A_n)\right)\mathbf{x}^{k+1} & = &
    \left(\alpha P_n-\mathrm{Re}(A_n)\right)\mathbf{x}^{k+{1\over 2}}+\mathbf{b}
     \end{array}
\right.
\end{equation}
Clearly, the last formulation is the most interesting from a practical point of view,
since it does not involve any inverse matrix and it easily allows to define an inexact formulation of
the quoted method: in principle, at each iteration, the PHSS method requires the exact solutions
with respect to the large matrices $\alpha P_n+\mathrm{Re}(A_n)$ and $\alpha P_n+\mathrm{i}\,
\mathrm{Im}(A_n)$. This requirement is impossible to achieve in practice and an inexact outer
iteration is computed by applying a Preconditioned Conjugate Gradient method (PCG) and a Preconditioned
Generalized Minimal Residual method (PGMRES), with preconditioner $P_n$, to the coefficient matrices
$\alpha P_n+\mathrm{Re}(A_n)$ and $\alpha P_n+\textrm{i}\ \mathrm{Im}(A_n)$, respectively.
Hereafter, we denote by IPHSS method the described inexact PHSS iterations. \newline The accuracy for
the stopping criterion of these additional inner iterative procedures must be chosen by taking
into account the accuracy obtained by the current step of the outer iteration (see \cite{BGN-SIMAX-2003,BGST-NM-2005}
and section \ref{sez:numerical_tests} for some remarks about this topic).
The most remarkable fact is that the PHSS and IPHSS methods show the same convergence properties,
though the computational cost of the latter is substantially reduced with respect to the former.\newline
In the IPHSS perspective the spectral properties induced by the preconditioner $P_n$ are relevant
not only with respect to the IPHSS rate of convergence, i.e., the outer iteration, but also with respect
to the PCG and PGMRES ones. So, the spectral analysis of the matrix sequence $\{P_n^{-1}\mathrm{Im} (A_n)\}$
becomes relevant exactly as the spectral analysis of the matrix sequence $\{P_n^{-1} \mathrm{Re} (A_n)\}$.\newline
Lastly, it is worth stressing that a deeper insight of the HSS/PHSS convergence properties
with respect to the skew-Hermitian part pertains to the framework of multi-iterative methods
\cite{S-CAM-1993}, among which multigrid methods represent a classical example.
Typically, a multi-iterative method is composed by two, or more, different iterative techniques, where
each one is cheap and potentially slow convergent. Nevertheless, these iterations have a
complementary spectral behavior, so that their composition becomes fast convergent. In fact,
the PHSS matrix iteration can be reinterpreted as the composition of two distinct iteration
matrices and the strong complementarity of these two components makes the contraction factor
of the whole procedure much smaller than the contraction factors of the two distinct components.
In particular, the skew-Hermitian contributions in the iteration matrix can have a role in accelerating
the convergence. The larger is $\b(\mathbf{x})$, i.e., the problem is convection dominated, the
more the FE matrix $A_n$ departs from normality. Thus, a stronger ``mixing up effect'' is
observed and the real convergence behavior of the method is much faster compared with
the forecasts of Theorem \ref{th:main}. See \cite{BGST-NM-2005} for further details.\newline
\section{Preconditioning strategy} \label{sez:preconditioning}
%
According to the notations and definitions in Section \ref{sez:fem}, 
the algebraic equations in (\ref{eq:modello_discreto}) can be rewritten in
matrix form as the linear system
\begin{equation*}
A_n(\a,\b) \mathbf{x} = \mathbf{b},
\end{equation*}
with
\begin{equation} \label{eq:def_A}
A_n(\a,\b)=\Theta_n(\a)+\Psi_n(\b)\in \mathbb{R}^{n\times n}, \
n=n(h),
\end{equation}
%
%
where $\Theta_n(\a)$ and $\Psi_n(\b)$ represent the approximation of the diffusive
term and approximation of the convective term, respectively. More precisely, we have
\begin{eqnarray}
(\Theta_n(\a))_{i,j} &=& \int_\Omega \a \nabla \varphi_i \nabla
\varphi_j \label{eq:def_theta}\\
(\Psi_n(\b))_{i,j}&=& -\int_\Omega (\nabla \varphi_i \cdot \b) \
\varphi_j, \label{eq:def_psi}
\end{eqnarray}
where suitable quadrature formula are considered in the case of non constant coefficient functions
$\a$ and $\b$.\newline
Thus, according to (\ref{eq:HSS_splitting}), the Hermitian/skew-Hermitian decomposition
of $A_n(\a,\b)$ is given by
\begin{eqnarray}
\mathrm{Re}(A_n(\a,\b)) &=& \Theta_n(\a) +
\mathrm{Re}(\Psi_n(\b)), \label{eq:def_ReA}\\
\mathrm{i}\,  \mathrm{Im}(A_n(\a,\b)) &=& \mathrm{i}\,
\mathrm{Im}(\Psi_n(\b)), \label{eq:def_ImA}
\end{eqnarray}
where
\begin{eqnarray*}
\mathrm{Re}(\Psi_n(\b)) &=&
\frac{1}{2}(\Psi_n(\b)+\Psi_n^T(\b))=E_n(\b),\\
\mathrm{i}\,  \mathrm{Im}(\Psi_n(\b)) &=& \Psi_n(\b)-E_n(\b),
\end{eqnarray*}
since by definition, the diffusion therm $\Theta_n(\a)$ is a Hermitian matrix and
does not contribute to the skew-Hermitian part of $A_n(\a,\b)$.
Notice also that $E_n(\b)=0$ if $\mathrm{div}(\b)=0$.\newline
Clearly, the quoted HSS decomposition can be performed on any single elementary matrix
related to $\mathcal{T}_h$ by considering the standard assembling procedure.\newline
By construction, the matrix $\mathrm{Re}(A_n(\a,\b))$ is symmetric and positive definite
whenever $\lambda_{\min}(\Theta_n(\a)) \ge \rho(E_n(\b))$. Indeed, without the condition
$\mathrm{div}(\b)\ge 0$, the matrix $E_n(\b)$ does not have a definite sign.\newline
Moreover, the Lemma below allows to obtain further information regarding such a structural
assumption.\newline
\begin{lemma}\label{lemma:normaE}
Let $\{E_n(\b)\}$ be the matrix sequence defined as
\[
E_n(\b)=\frac{1}{2}(\Psi_n(\b)+\Psi_n^T(\b)).
\]
%
%
Under the assumptions in \emph{(\ref{eq:ipotesi_coefficienti})}, then it holds
\[
\|E_n(\b)\|_2  \le \|E_n(\b)\|_\infty \le Ch^2,
\]
with $C$ absolute positive constant only depending on $\b(\mathbf{x})$ and $\Omega$.
%
%
%
\end{lemma}
\begin{proof}
By applying the Green formula, it holds  that for any $i,j=1, \ldots, n(h)$
\begin{eqnarray*}
(E_n(\b))_{i,j}&=& -\frac{1}{2}\int_\Omega \left ( (\nabla
\varphi_i \cdot \b)\varphi_j +  (\nabla \varphi_j \cdot
\b)\varphi_i \right) = -\frac{1}{2} \int_\Omega \b \cdot \nabla
(\varphi_i \varphi_j)
\\
&=&\frac{1}{2} \int_\Omega  \mathrm{div}(\b) \cdot
\varphi_i \varphi_j \\
&=& \frac{1}{2} \sum_{K\subseteq S(\varphi_i) \cap S(\varphi_j) }
\int_{K} \mathrm{div}(\b)
\varphi_i \varphi_j, \\
\end{eqnarray*}
with respect to the related mesh $\mathcal{T}_h=\{K\}$ and where
$S(\varphi_k)$ denotes the support of basis element $\varphi_k$ on $\mathcal{T}_h$. Thus, we have
\[
|E_n(\b)|_{i,j}\le \frac{1}{2} \sup_{\Omega} \left| \mathrm{div}
(\b)\right| \sum_{K\subseteq S(\varphi_i) \cap S(\varphi_j) }
\int_{K} |\varphi_i \varphi_j|
\le \frac{1}{4} \sup_{\Omega} \left| \mathrm{div} (\b)\right| qh^2
\]
since $|\varphi_k| \le 1$ for any $k$ and
\[
\sum_{K \subseteq S(\varphi_i) \cap S(\varphi_j) } \int_{K}
|\varphi_i \varphi_j| \le \frac{h^2}{2} \# \{K \in \mathcal{T}_h |
K\subseteq S(\varphi_i) \cap S(\varphi_j)\}\le \frac{q h^2}{2},
\]
where $h$ is the finesse parameter of the mesh $\mathcal{T}_h$
and $q$ equals the maximum number of involved mesh elements with respect to the mesh sequence
$\{\mathcal{T}_h\}$.\newline
Lastly, since $E_n(\b)$ is a Hermitian matrix, it holds
\[
\|E_n(\b)\|_2  \le \|E_n(\b)\|_\infty \le D
\max_{i,j=1,\ldots,n(h)} |E_n(\b)|_{i,j} \le \frac{1}{4}
D\sup_{\Omega} | \mathrm{div} (\b)| qh^2,
\]
where $D$ denotes the maximum number of nonzero entries on the rows of $E_n(\b)$.
%
%
\end{proof}
\ \newline
\begin{remark} 
The claim of Lemma \emph{\ref{lemma:normaE}} holds whenever a quadrature formula
with error $O(h^2)$ is consider for approximating the integrals involved in \emph{(\ref{eq:def_psi})}. \newline
\end{remark}
\ \newline
Moreover, in the special case of a structured uniform mesh
on $\Omega=(0,1)^2$ as considered in the next section, under the assumptions of Lemma \ref{lemma:normaE}
and $\a(\mathbf{x})\ge \a_0 >0$, we have $\Theta_n(\a) \ge c h^2 I_n$ (with $c$ absolute positive
constant), so that $\mathrm{Re}(A_n(\a,\b)) \ge (c-C)h^2 I_n$.
Here, we are referring to the standard ordering relation between Hermitian matrices,
i.e., the notation $X \ge Y$, with $X$ and $Y$ Hermitian matrices, means that $X - Y$
is nonnegative definite.
Thus, under the assumption that $|\mathrm{div}(\b)|$ is smaller than a positive suitable
constant, it holds that $\mathrm{Re}(A_n(\a,\b))$ is real, symmetric, and positive definite. \newline
However, the main drawback is due to ill-conditioning, since the condition number is asymptotic
to $h^{-2}$, so that preconditioning is highly recommended.\newline
By referring to a preconditioning strategy previously analyzed in the case of FD
approximations \cite{S-NM-1999,ST-LAA-1999,ST-ETNA-2000,ST-ETNA-2003,ST-SIMAX-2003}
or FE approximations \cite{ST-NA-2001} of the diffusion equation and
recently applied to FD approximations \cite{BGST-NM-2005} of (\ref{eq:modello}),
we consider the preconditioning matrix sequence $\{P_n(\a)\}$  defined as
\begin{equation} \label{eq:def_P}
P_n(\a)=D_n^{\frac{1}{2}}(\a) A_n(1,0) D_n^{\frac{1}{2}}(\a)
\end{equation}
where
$D_n(\a)=\mathrm{diag}(A_n(\a,0))\mathrm{diag}^{-1}(A_n(1,0))
$, i.e., the suitable scaled main diagonal of
$A_n(\a,0)$ and clearly $A_n(\a,0)$ equals $\Theta_n(\a)$.\newline
In such a way, the solution of the linear system $A_n(\a,\b)\mathbf{u}=\mathbf{f}$ is reduced to computations
involving diagonals and the matrix $A_n(1,0)$. In the case of uniform structured mesh this task
can be efficiently performed by considering fast Poisson solvers,
such as those based on the cyclic reduction idea (see e.e.
\cite{BDGG-SINUM-1971,D-SIAMRev-1970,S-SIAMRev-1977}) and several specialized multigrid methods
(see e.g. \cite{H-Springer-1985,S-NM-2002}). \newline
It is worth stressing that the preconditioner is tuned only with respect to the diffusion matrix $\Theta_n(\a)$
owing to the PHSS convergence properties highlighted in section \ref{sez:PHSS}. Indeed, the
PHSS shows a convergence behavior mainly depending on the spectral properties of the matrix
$P_n^{-1}(\a)\mathrm{Re}(A_n(\a,\b))$. Nevertheless, the skew-Hermitian contribution may play a role
in speeding up considerably the convergence (see \cite{BGST-NM-2005} for further details). \newline
Hereafter, we denote by $\{A_n(\a,\b)\}$, $n=n(h)$ the matrix sequence associated to a family of
mesh $\{T_h\}$, with decreasing finesse parameter $h$.
As customary, the whole preconditioning analysis will refer to a matrix sequence instead to
a single matrix, since the goal is to quantify the difficult of the linear system resolution in
relation to the accuracy of the chosen approximation scheme.
\section{Spectral analysis and clustering properties in the case of structured uniform meshes} \label{sez:clustering}
In the present section we analyze the spectral properties of the preconditioned matrix sequences
\[
\{P_n^{-1}(\a)\mathrm{Re}(A_n(\a,\b))\} \ \mathrm{and}\
\{P_n^{-1}(\a)\mathrm{Im}(A_n(\a,\b))\}
\]
in the special case of $\Omega=(0,1)^2$ with a structured uniform mesh as in Figure
\ref{fig:mesh_strutturata_uniforme}, by using the spectral tools derived from Toeplitz theory
\cite{BG-LAA-1998,BS-Springer-1998,S-LAA-1998,S-LAA-1998-2}. The aim is to prove the optimality of the
PHSS method, i.e., the PHSS iterations number for reaching the solution within a fixed accuracy can be
bounded from above by a constant independent of the dimension $n=n(h)$. A more in depth analysis is also
considered for foreseeing the IPHSS method convergence behavior.\newline
\begin{figure}
\centering
\epsfig{file=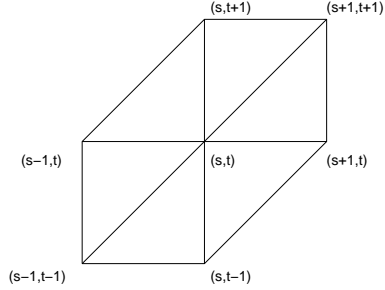,height=6cm}
\vskip -1.5cm
\caption{Uniform structured mesh giving rise to a Toeplitz matrix in the constant coefficient case $\a(\mathbf{x})=1$.}
\label{fig:mesh_strutturata_uniforme}
\end{figure}
We make reference to the following definition.\newline
\begin{definition} \emph{\cite{T-LAA-1996}} \label{def:cluster}
Let $\{A_n\}$ be a sequence of matrices of increasing dimensions
$n$ and let $g$ be a measurable function defined over a set $K$ of
finite and positive Lebesgue measure. The sequence $\{A_n\}$ is distributed as
the measurable function $g$ in the sense of the eigenvalues, i.e.,
$\{A_n\}\sim_\lambda g$ if, for every $F$ continuous, real valued
and with bounded support, we have
\begin{equation*} 
\lim_{n\to\infty} \frac 1 {n}\sum_{j=1}^{n}
F\left(\lambda_j\bigl(A_n\bigr)\right)=
 {1\over m\{K\}}\int_{K} F(\theta(s))\, ds,
\end{equation*}
where $\lambda_j(A_n)$, $j=1,\ldots,n$, denote the eigenvalues
of $A_n$.\newline
The sequence $\{A_n\}$ is clustered at $p$ if it is distributed as
the constant function $g(x)\equiv p$, i.e., for any $\varepsilon
>0$,
$ 
\# \left\{i \left.\right|\lambda_i(A_n) \notin (p - \varepsilon, p
+ \varepsilon)\right\}=o(n)$.
%
%
The sequence $\{A_n\}$ is properly (or strongly) clustered at $p$
if for any $\varepsilon> 0$ the number of the eigenvalues of $A_n$
not belonging to $(p - \varepsilon, p + \varepsilon)$ can be
bounded by a pure constant eventually depending on $\varepsilon$,
but not on $n$.
\end{definition}
%
%
\ \newline \newline
First, we analyze the spectral properties of the matrix sequence $\{P_n^{-1}(\a)\mathrm{Re}(A_n(\a,\b))\}$
that directly influences the convergence behavior of the PHSS method. We refer
to a preliminary result in the case in which the convection term is not present. \newline
\begin{theorem}\emph{\cite{ST-NA-2001}} \label{teo:clu+se_b0}
%
%
Let $\{A_n(\a,\mathbf{0})\}$ and $\{P_n(\a)\}$ be the Hermitian
positive definite matrix sequences defined according to
\emph{(\ref{eq:def_A})} and \emph{(\ref{eq:def_P})}. If the
coefficient $\a(\mathbf{x})$ is strictly positive and belongs to
${\bf C}^2(\overline \Omega)$, then
the sequence $\{P_n^{-1}(\a) A_n(\a,\mathbf{0})\}$ is properly
clustered at $1$.
Moreover, for any $n$ all the eigenvalues of $P_n^{-1}(\a) A_n(\a,\mathbf{0})$
belong to an interval $[d,D]$ well separated from zero {\em
[Spectral equivalence property]}.
\ \newline
\end{theorem}
\ \newline
The extension of this claim in the case of the matrix
sequence $\{\mathrm{Re}(A_n(\a,\b))\}$ with
$\mathrm{Re}(A_n(\a,\b)) \ne \Theta_n(\a)$ can be proved
under the additional assumptions of Lemma \ref{lemma:normaE}. \newline
\begin{theorem} \label{teo:clu+se_ReA}
Let $\{\mathrm{Re}(A_n(\a,\b))\}$ and $\{P_n(\a)\}$ be the
Hermitian positive definite matrix sequences defined according to
\emph{(\ref{eq:def_ReA})} and \emph{(\ref{eq:def_P})}. 
%
%
Under the assumptions in \emph{(\ref{eq:ipotesi_coefficienti})}, then
the sequence $\{P_n^{-1}(\a) \mathrm{Re}(A_n(\a,\b))\}$ is
properly clustered at $1$.
Moreover, for any $n$ all the eigenvalues of $P_n^{-1}(\a)
\mathrm{Re}(A_n(\a,\b))$ belong to an interval $[d,D]$ well
separated from zero {\em [Spectral equivalence property]}.
\end{theorem}
\begin{proof}
%
The proof technique refers to a previously analyzed FD case \cite{ST-ETNA-2003} and it is extended
for dealing with the additional contribution given by $E_n(\b)$.\newline
First, we consider the spectral equivalence property.
Due to a similarity argument, we analyze the sequence $\{
\Theta_n^{-1}(1) (\Theta_n^*(a)+E_n^*(\b))\}$, where
$X^*=D_n^{-\frac{1}{2}}(a)XD_n^{-\frac{1}{2}}(a)$
and
$D_n(\a)=\mathrm{diag}(\Theta_n(\a))\mathrm{diag}^{-1}(\Theta_n(1))$.
According to the assumptions and Theorem 7 in \cite{ST-NA-2001}, we have the
following asymptotic expansion
\begin{equation} \label{eq:asymptotic_expansion}
\Theta_n^*(\a)=\Theta_n(1) +h^2 F_m(\a)+o(h^2)G_n(\a)
\end{equation}
so that
\[
\Theta_n^{-1}(1)\Theta_n^*(\a)=I_n+\Theta_n^{-1}(1)(h^2 F_n(\a)
+o(h^2)G_n(\a)).
\]
Taking into account the order of the zeros of the generating
function of the Toeplitz matrix $\Theta_n(1)$, we infer that there
exists a constant $c_1$ so that $\|\Theta_n^{-1}(1)\|_2
\le c_1 h^{-2}$ \cite{BG-LAA-1998,S-LAA-1998}. Moreover, it also holds
that
\begin{equation} \label{eq:norma2_Estarb}
\|E_n^*(\b)\|_2 \le \frac{C h^2}{\min_{\Omega}\a}.
\end{equation}
Therefore, by standard linear algebra, we have
\begin{eqnarray*}
\lambda_{\max}(\Theta_n^{-1}(1) (\Theta_n^*(\a)+E_n^*(\b)) & \le &  \|\Theta_n^{-1}(1) (\Theta_n^*(\a)+E_n^*(\b))\|_2 \\
& \le & \|I_n\|_2 + h^2\|\Theta_n^{-1}(1)\|_2\|F_n(\a)+o(1)G_n(\a)\|_2\\
& & \quad +  \|\Theta_n^{-1}(1)\|_2\|E_n^*(\b)\|_2\\
& \le & 1+c_1\left (\|F_n(\a)\|_2+o(1)\|G_n(\a)\|_2
+\frac{C}{\min_{\Omega} \a} \right),
\end{eqnarray*}
where $\|F_n(\a)\|_2$ and $\|G_n(\a)\|_2$ are uniformly bounded
since $F_n(\a)$ and $G_n(\a)$ are bounded symmetric sparse
%
%
matrices by virtue of Theorem 7 in \cite{ST-NA-2001}.\newline
Conversely, a bound from below for
$\lambda_{\min}(\Theta_n^{-1}(1)(\Theta_n^*(\a)+E_n^*(\b))$
requires a bit different technique with respect to
Theorem \ref{teo:clu+se_b0}, owing to the presence
of the convective term. Again by a similarity argument and by
referring to the Courant-Fisher Theorem, we have
\begin{eqnarray*}
\lambda_{\min}(\Theta_n(1)^{-1}(\Theta_n^*(\a)+E_n^*(\b))&=&
\min_{x \ne 0} \frac{\mathbf{x}^T
\Theta_n^{-\frac{1}{2}}(1)(\Theta_n^*(\a)+E_n^*(\b))\Theta_n^{-\frac{1}{2}}(1)\mathbf{x}}{\mathbf{x}^T\mathbf{x}}\\
&\ge & \lambda_{\min}(\Theta_n^{-1}(1)\Theta_n^*(\a)) \\
&& \quad + \min_{x \ne 0} \frac{\mathbf{x}^T
\Theta_n^{-\frac{1}{2}}(1)E_n^*(\b)\Theta_n^{-\frac{1}{2}}(1)\mathbf{x}}{\mathbf{x}^T\mathbf{x}}
\end{eqnarray*}
where
\[
\lambda_{\min}(\Theta_n^{-1}(1)\Theta_n^*(\a)) \ge \left( 1+
\tilde{c} \frac{\max_\Omega \a}{\min_\Omega \a} (\|F_n(\a)\|_2
+o(1)\|G_n(\a)\|_2) \right)^{-1}
\]
as proved in \cite{ST-NA-2001,ST-ETNA-2003} and
\[
\min_{x \ne 0} \frac{\mathbf{x}^T
\Theta_n^{-\frac{1}{2}}(1)E_n^*(\b)\Theta_n^{-\frac{1}{2}}(1)\mathbf{x}}{\mathbf{x}^T\mathbf{x}}
\ge \frac{\lambda_{\min}(E_n^*(\b))}{\lambda_{\min}(\Theta_n(1))}
\ge   - \frac{Ch^2}{c_2h^2\min_\Omega \a} =  - \frac{C}{c_2
\min_\Omega \a},
\]
being $\lambda_{\min}(\Theta_n(1))\le c_2h^2$.\newline
%
%
The proof of the presence of a proper cluster again makes use of
a double similarity argument and of the asymptotic expansion
in (\ref{eq:asymptotic_expansion}), so that we analyze the
spectrum of the matrices
\[
X_n=I_n+ \Theta_n^{-\frac{1}{2}}(1) (h^2F_n(\a)+o(h^2)G_n(\a) +
E_n^*(\b) ) \Theta_n^{-\frac{1}{2}}(1)
\]
similar to the matrices
$\Theta_n^{-1}(1)(\Theta_n^*(\a)+E_n^*(\b))$.\newline
As in the case of Theorem \ref{teo:clu+se_b0}, we refer to the matrix
$U \in \mathbb{R}^{n\times p}$, whose columns are made up by
considering the orthonormal eigenvectors of $\Theta_n(1)$
corresponding to the eigenvalues $\lambda_i(\Theta_n(1)) \ge
\lceil \varepsilon_n^{-1} \rceil h^2$, since we know
\cite{ST-ETNA-2003} that for any sequence $\{\varepsilon_n\}$
decreasing to zero (as slowly as wanted) it holds
\[
\# \mathcal{I}_{\varepsilon_n} = O(\lceil
\varepsilon_n^{-1}\rceil), \quad
\mathcal{{I}}_{\varepsilon_n}= \{ i | \lambda_i(\Theta_n(1)) <
\lceil \varepsilon_n^{-1} \rceil h^2 \}.
\]
Thus, we consider the following split projection
\begin{eqnarray*}
U^T X_n U= I_p +Y_p +Z_p +W_p
\end{eqnarray*}
with $Y_p=h^2 D_p U^T F_n U D_p$, $Z_p=o(h^2) D_p U^T G_n U D_p$
and $W_p=h^2 D_p U^T E_n^*(\b) U D_p$, where
$D_p=\mathrm{diag}\left(\lambda_i^{-\frac{1}{2}}\right)_{i \notin
\mathcal{{I}}_{\varepsilon_n}} \in \mathbb{R}^{p\times p}$. The
matrices $Y_p +Z_p +W_p$ are of infinitesimal spectral norm since
all the terms $Y_p$, $Z_p$ \cite{ST-NA-2001} and $W_p$ are too. In
fact, by virtue of (\ref{eq:norma2_Estarb}) and the $D_p$ matrix
definition we have
\[
\|W_p\|_2 \le \|D_p\|_2^2 \|E_n^*(\b)\|_2 \le \frac{1}{\lceil
\varepsilon_n^{-1}\rceil h^2} \frac{C h^2}{\min_{\Omega}\a} \le
\frac{C}{\min_{\Omega}\a} \lceil \varepsilon_n \rceil .
\]
Therefore, by applying the Cauchy interlacing theorem
\cite{GVL-1996}, 
it directly follows that for
any $\varepsilon>0$ there exists $\bar{n}$ such that for any
$n>\bar{n}$ (with respect to the
%
%
the ordering induced by the considered mesh family) at least $n-
O(\lceil \varepsilon_n^{-1}\rceil)$ eigenvalues of the
preconditioned
matrix  
belong to the open interval $(1-\varepsilon,1+\varepsilon)$.
\end{proof}
\ \newline
\begin{remark}
The claim of Theorem \emph{\ref{teo:clu+se_ReA}} holds both in the case in which
the matrix elements in \emph{(\ref{eq:def_theta})}
and \emph{(\ref{eq:def_psi})} are evaluated exactly and whenever a
quadrature formula with error $O(h^2)$ is consider to approximate
the involved integrals. \newline
\end{remark}
\ \newline
The previous results prove the optimality both of the PHSS method and
of the PCG, when applied in the IPHSS method for the inner iterations. However, still in the
case of the IPHSS method, suitable spectral properties of the preconditioned matrix sequence
$\{P_n^{-1}(\a)\mathrm{Im}(A_n(\a,\b))\}$ have to be proven with respect to the PGMRES
application. \newline
Hereafter, even if we make direct reference to the spectral Toeplitz theory, we prefer
to preliminary analyze the matrices by considering the standard FE assembling procedure.
Indeed, the FE elementary matrices suggest the \emph{local domain analysis} approach
in a more natural way than in the FD case.
More precisely, this purely linear algebra technique consists in an additive decomposition
of the matrix in simpler matrices allowing a proper majorization for any single term
(see also \cite{BGST-NM-2005,BS-SINUM-2007}).
\ \newline
\begin{theorem} \label{teo:clu+sb_ImA}
Let $\{\mathrm{Im}(A_n(\a,\b))\}$ and $\{P_n(\a)\}$ be the
Hermitian matrix sequences defined according to
\emph{(\ref{eq:def_ImA})} and \emph{(\ref{eq:def_P})}.
%
%
Under the assumptions in \emph{(\ref{eq:ipotesi_coefficienti})}, then 
the sequence $\{P_n^{-1}(\a) \mathrm{Im}(A_n(\a,\b))\}$ is
spectrally bounded and properly clustered at $0$ with respect to
the eigenvalues.
\end{theorem}
\begin{proof}
Clearly, by referring to the standard assembling procedure, the Hermitian matrix
$\mathrm{Im}(A_n(\a,\b))$ can be represented as
\begin{equation*}
\mathrm{Im}(A_n(\a,\b)) = \sum_{K\in \mathcal{T}_h}
\mathrm{Im}(A_{n}^{K}(\a,\b))
\end{equation*}
with the elementary matrix corresponding to $\mathrm{Im}(A_{n}^{K}(\a,\b))$ given by
\begin{equation*}
\mathrm{Im}(A_{el}^{K}(\a,\b)) = - \frac{\mathrm{i}}{2} \left[
\begin{array}{ccc}
0 & -(\gamma_{12}-\gamma_{21}) & -(\gamma_{13}-\gamma_{31}) \\
\gamma_{12}-\gamma_{21}& 0 & -(\gamma_{23}-\gamma_{32}) \\
\gamma_{13}-\gamma_{31} & \gamma_{23}-\gamma_{32} & 0 \\
\end{array}
\right].
\end{equation*}
Here $\gamma_{ij}=\gamma(\varphi_i, \varphi_j) = \int_{K} (\nabla
\varphi_i \cdot \b) \ \varphi_j$ and the indices $i,j$ refer to a
local counterclockwise numbering of the nodes of the generic
element $K\in \mathcal{T}_h$.\newline
With respect the desired analysis
this matrix can be split, in turn, as
\begin{equation*}
\mathrm{Im}(A_{el}^{K}(\a,\b)) = - \frac{1}{2} \left(
(\gamma_{12}-\gamma_{21})  A_{el}^{K,1} +
(\gamma_{23}-\gamma_{32}) A_{el}^{K,2}+
(\gamma_{13}-\gamma_{31}) A_{el}^{K,3}
\right),
\end{equation*}
where
\begin{equation*}
A_{el}^{K,1}= \mathrm{i} \left[
\begin{array}{ccc}
0 & -1 & 0 \\
1 & 0 & 0 \\
0 & 0 & 0 \\
\end{array}
\right],\
A_{el}^{K,2}= \mathrm{i} \left[
\begin{array}{ccc}
0 & 0 & 0 \\
0 & 0 & -1 \\
0 & 1& 0 \\
\end{array}
\right],\
A_{el}^{K,3}= \mathrm{i}\left[
\begin{array}{ccc}
0 & 0 & -1 \\
0 & 0 & 0 \\
1 & 0 & 0 \\
\end{array}
\right].
\end{equation*}
By virtue of a permutation argument, the immersion of each matrix
$A_{el}^{K,r}$, $r=1,\dots,3$, into the original matrix $\mathrm{Im}(A_n(\a,\b))$
can be associated to a null matrix of dimension $n(h)$ except for
the $2$ by $2$ block given by $A_{el}^{K,1}$ in diagonal position.
Therefore, we can refer to the same technique considered in
\cite{BGST-NM-2005}. More precisely, it holds that
\[
\pm \mathrm{i} \left[
\begin{array}{cc}
0 & -1  \\
1 & 0  \\
\end{array}
\right] \le w \left [ \begin{array}{cc}
1+\hat{h}^2 & -1  \\
-1 & 1+\hat{h}^2  \\
\end{array}
\right]
\]
provided that $w \ge (\hat{h} \sqrt{\hat{h}^2+2})^{-1}$ and also
for any constant $\tilde{\gamma}$
\[
\pm \gamma \, A_{el}^{K,1}= \pm \tilde{\gamma} \, \mathrm{i}
\left[
\begin{array}{ccc}
0 & -1 & 0 \\
1 & 0 & 0 \\
0 & 0 & 0 \\
\end{array}
\right] \le |\tilde{\gamma} |\, w \left [ \begin{array}{ccc}
1+\hat{h}^2 & -1  & 0\\
-1 & 1+\hat{h}^2  & 0\\
0 & 0  & 0\\
\end{array}
\right]  =  |\tilde{\gamma} |\, w \, R_{el}^{K,1}.
\]
Thus, by linearity and positivity, we infer that
\begin{eqnarray*}
\pm \mathrm{Im}(A_{el}^{K}(\a,\b)) &\le& \frac{w}{2} \left (
|\gamma_{12}-\gamma_{21}|  R_{el}^{K,1}  +
|\gamma_{23}-\gamma_{32}|  R_{el}^{K,2} +
|\gamma_{13}-\gamma_{31}| R_{el}^{K,3}  \right )\\
&\le & w h \|\b\|_\infty
\left[
\begin{array}{ccc}
2(1+\hat{h}^2) & -1 & -1 \\
-1 & 2(1+\hat{h}^2) & -1 \\
-1 & -1 & 2(1+\hat{h}^2) \\
\end{array}
\right]
\end{eqnarray*}
since
\[
R_{el}^{K,1} =  \left[
\begin{array}{ccc}
1+\hat{h}^2 & -1 & 0 \\
-1 & 1+\hat{h}^2 & 0 \\
0 & 0 & 0 \\
\end{array}
\right],
R_{el}^{K,2} =  \left[
\begin{array}{ccc}
0 & 0 & 0 \\
0 & 1+\hat{h}^2 & -1 \\
0 & -1& 1+\hat{h}^2 \\
\end{array}
\right],
R_{el}^{K,3}  \left[
\begin{array}{ccc}
1+\hat{h}^2 & 0 & -1 \\
0 & 0 & 0 \\
-1 & 0 & 1+\hat{h}^2 \\
\end{array}
\right]
\]
and $|\gamma_{ij}-\gamma_{ji}|\le 2 h \|\b\|_\infty$ for any $i,j$.\newline
Therefore, with respect to the matrices of dimension $n(h)$, we
obtain the following key majorization
\begin{eqnarray*} \pm \mathrm{Im}(A_n(\a,\b)) &\le &
w h \|\b\|_\infty (c \Theta_n(1)+ 2\hat{h}^2 I_n), 
\end{eqnarray*}
so that by virtue of LPOs and spectral equivalence properties
proved in Theorem \ref{teo:clu+se_ReA} it holds that
\begin{equation*}
\pm \mathrm{Im}(A_n(\a,\b)) \le
w h \|\b\|_\infty \left(\frac{c}{{\min_\Omega \a}}\Theta_n(\a)+ 2\hat{h}^2  I_n\right)
\le
w \tilde{c} h \|\b\|_\infty \left(P_n(\a)+ \check{c} \hat{h}^2  I_n\right).
\end{equation*}
In such a way the claim can be obtained according to the very same
proof technique considered in \cite{BGST-NM-2005}. In fact, by
considering the symmetrized preconditioned matrices and by
referring to the Courant-Fisher theorem, we find
\[
\pm P_n^{-\frac{1}{2}}(\a) \mathrm{Im}(A_n(\a,\b)) P_n^{-\frac{1}{2}}(\a)
\le  w \tilde{c} h \|\b\|_\infty (I_n+\check{c} \hat{h}^2  P_n^{-1}(\a))
\]
and the spectral analysis must focus on the spectral properties of
the Hermitian matrix sequences
$\{ w h (I_n+ \check{c} \hat{h}^2  P_n^{-1}(\a)) \}$.\newline
By referring to Toeplitz spectral theory, the spectral boundeness
property is proved since we simply have
\[
\lambda_{\max}(w h (I_n+ \check{c} \hat{h}^2 P_n^{-1}(\a))) \le
w h \left(1+\gamma \frac{\hat{h}^2}{h^2}\right) = \varepsilon
+ \gamma \varepsilon^{-1}
\]
independent of $h$, if we choose $w=\hat{h}^{-1}$ and $\hat{h}=h \varepsilon^{-1}$ for any given $\varepsilon >0$. \\
Moreover, thanks to the strictly positivity of $\a(\mathbf{x})$
and for the same choice of $w$ and $\hat{h}$, we infer that
\begin{eqnarray*}
\lambda_i(w h (I_n+ \check{c} \hat{h}^2 P_n^{-1}(\a)))
&=& w h (1+ \check{c} \hat{h}^2 \lambda_s^{-1}(P_n(\a))
), \quad s=n+1-i\\
&\le& \varepsilon+ \frac{\check{c}}{\min_\Omega \a}
\frac{h^2}{\varepsilon} \lambda_s^{-1}(\Theta_n(1))).
\end{eqnarray*}
Since $\# \{s \ | \ \lambda_s((\Theta_n(1))) < {h^2}/{\varepsilon}
\} = O( \varepsilon^{-1})$  the proof is over.
\end{proof}
\ \newline
\begin{remark}
The claim of Theorem \emph{\ref{teo:clu+sb_ImA}} holds also in the
case in which the matrix elements in \emph{(\ref{eq:def_psi})} are
evaluated by applying any quadrature formula with error $O(h^2)$.\newline
%
%
\end{remark}
\ \newline
To sum up, with the choice $\alpha=1$ and under the regularity assumptions
of Theorems \ref{teo:clu+se_ReA} and \ref{teo:clu+sb_ImA}, we have proven that
the PHSS method is optimally convergent (linearly, but with a convergence independent
of the matrix dimension $n(h)$ due to the spectral equivalence property). In addition,
when considering the IPHSS method, the PCG converges superlinearly owing to the
proper cluster at $1$ of the matrix sequence $\{P_n^{-1}(\a)\mathrm{Re}(A_n(\a,\b))\}$,
that clearly induces a proper cluster at $1$ for the matrix sequence
$\{I+P_n^{-1}(\a)\mathrm{Re}(A_n(\a,\b))\}$. \newline
In an analogous way, the spectral boundeness and the proper clustering at
$0$ of the matrix sequence $\{P_n^{-1}(\a)\mathrm{Im}(A_n(\a,\b))\}$ allow to claim
that PGMRES in the inner IPHSS iteration converges superlinearly when applied
to the coefficient matrices $\{I+\mathrm{i}P_n^{-1}(\a)\mathrm{Im}(A_n)\}$.\newline
\section{Numerical tests} \label{sez:numerical_tests}
Before analyzing in detail the obtained numerical results, we wish to
give a general description of the performed numerical experiments and of some implementation details.
The case of unstructured meshes is discussed at the end of the section, while further remarks on
the computational costs are reported in the next section. \newline
Hereafter, we have applied the PHSS method, with the preconditioning strategy
described in Section \ref{sez:preconditioning}, to FE approximations of the problem
(\ref{eq:modello}). First, we consider the case of $\a$ uniformly positive function and
$\b$  function vector which are regular enough as required by Lemma \ref{lemma:normaE}
and Theorems \ref{teo:clu+se_ReA} and \ref{teo:clu+sb_ImA}. The domain of integration $\Omega$ is the
simplest one, i.e., $\Omega=(0,1)^2$ and we assume Dirichlet boundary conditions.
Whenever required, the involved integrals have been approximated by means of the middle point rule (the
approximation by means of the trapezoidal rule gives rise to analogous results).\newline
As just outlined in Section \ref{sez:PHSS}, the preconditioning strategy  has been tested
by considering the PHSS formulation reported in (\ref{eq:PHSS-P}) and by applying the PCG and
PGMRES methods for the Hermitian and skew-Hermitian inner iterations, respectively.
Indeed, in principle each iteration of the PHSS method requires the exact solutions with
large matrices as defined in (\ref{eq:PHSS}), which can be impractical
in actual implementations. Thus, instead of inverting the matrices
$\alpha I+P_n^{-1}(a)\mathrm{Re}(A_n(\a,\b))$ and $\alpha I +P_n^{-1}(a)\mathrm{Im}(A_n(\a,\b))$,
the PCG and PGMRES are applied for the solution of system with coefficient
matrices $\alpha P_n(a)+\mathrm{Re}(A_n(\a,\b))$ and $\alpha P_n(a)+\mathrm{Im}(A_n(\a,\b))$,
respectively, and with $P_n(a)$ as preconditioner.
\newline
Preliminary, the numerical tests have been performed by setting the tolerances required by the inner
iterative procedures at the same value of the outer procedure tolerance. 
This allows a realistic check of the PHSS convergence properties in the case of the exact
formulation of the method.
Moreover, a significant reduction of the computational costs can be obtained
by considering the inexact formulation of the method, denoted in short as IPHSS.
In the IPHSS implementation of (\ref{eq:PHSS-P}), the inner iterations in  are switched to the
$(k+1)$--th outer step if
\begin{equation}
\frac{||r_{j,PCG}||_2}{||r_{k}||_2}\leq 0.1\,\eta^{k}, \quad
\frac{||r_{j,PGMRES}||_2}{||r_{k}||_2}\leq 0.1\,\eta^{k},
\label{eq:stop-crit}
\end{equation}
respectively, where $k$ is the current outer iteration, $\eta\in
(0,1)$,
and where $r_j$ is the residual at the $j$--th step of the present
inner iteration \cite{BGN-SIMAX-2003,BGST-NM-2005}. The reported results refer to the
case $\delta=0.9$, that typically gives the best performances.\newline
The quoted criterion is effective enough to show the behavior of the inner
and outer iterations for IPHSS. It allows to conclude that the IPHSS
and the PHSS methods have the same convergence features, but the
cost per iteration of the former is substantially reduced as evident from the
lower number of total inner iterations.\newline
It is worth stressing that more sophisticate stopping criteria may save a significant amount
of inner iterations with respect to (\ref{eq:stop-crit}). In particular, the
approximation error of the FE scheme could be taken into account to drive this tuning \cite{ANR-C-2001}.  \newline
Finally, mention has to be made to the choice of the parameter $\alpha$.
Despite the PHSS method is unconditionally convergent for any $\alpha>0$,
a suitable tuning, according to Theorem \ref{th:main}, can significantly reduce
the number of outer iterations. Clearly, the choice $\alpha=1$ is evident whenever
a cluster at $1$ of the matrix sequence $\{P_n^{-1}(a)\mathrm{Re}(A_n(\a,\b))\}$ is expected.
In the other cases, the target is to approximatively estimate the
optimal $\alpha$ value
\[
\alpha^*=\sqrt{\lambda_{\min}(P_n^{-1}\mathrm{Re}(A_n))\lambda_{\max}(P_n^{-1}\mathrm{Re}(A_n))}
\]
that makes the spectral radius of the PHSS iteration matrix bounded by
$\sigma(\alpha^*)={\sqrt{\kappa}-1}/{\sqrt{\kappa}+1}$, with
$\kappa={\lambda_{\max}(P_n^{-1}\mathrm{Re}(A_n))}/{\lambda_{\min}(P_n^{-1}\mathrm{Re}(A_n))}$
spectral condition number of $P_n^{-1}\mathrm{Re}(A_n)$, namely the Euclidean (spectral) condition number of the
symmetrized matrix.
%
%
\newline
All the reported numerical experiments are performed in Matlab,
with zero initial guess for the outer iterative solvers and
stopping criterion $||r_k||_2\leq 10^{-7}||r_0||_2$. \newline
No comparison is explicitly made with the case of the HSS method, since
the obtained results are fully comparable with those observed in the
FD approximation case \cite{BGST-NM-2005,C-Tesi-2006}. \newline
%
%
In Table \ref{tab:IT-ES1M1} we report the number of PHSS outer iterations required
to achieve the convergence for increasing values of the coefficient matrix size $n=n(h)$
when considering the FE approximation with the structured uniform mesh reported in
Figure \ref{fig:mesh_strutturata_uniforme} and with template function
$\a(x,y)=\a_1(x,y)=\exp(x+y)$, $\b(x,y)=[x\ y]^T$ satisfying the required regularity assumptions.
The averages per outer step for PCG and PGMRES iterations are also reported (the total is in brackets); the values refer to the case
of inner iteration tolerances that equals the outer iteration tolerance $tol=10^{-7}$. \newline
The numerical experiments plainly confirm the previous theoretical analysis
in Section \ref{sez:clustering}.
In particular, we observe that the outer convergence behavior does not depend on the coefficient
matrix dimension $n=n(h)$. The same holds true with respect to the PCG inner iterations, while
some dependency on $n$ is observed with respect to the PGMRES inner iterations. More precisely,
a higher number of PGMRES inner iterations is required in the first few steps of the PHSS outer iterations
for increasing $n$.
Nevertheless, a variable tolerance stopping criterion for the inner iterations as devised in the IPHSS
method is able to stabilize, or at least to strongly reduce, this sensitivity. \newline
The numerical results in Table \ref{tab:OUTLIER-ES1M1} give evidence of the strong clustering properties
when the previously defined preconditioner $P_n(\a)$ is applied. More precisely, for increasing
values of the coefficient matrix dimension $n$, we report the number of outliers of
$P_n^{-1}(\a)\mathrm{Re}(A_n(\a,\b))$ with respect to a cluster at $1$ with radius
$\delta = 0.1$ (or $\delta =0.01$): $m_-$ is the number of outliers less then $1-\delta$,
$m_+$ is the number of outliers greater then $1 +\delta$,  $p_{\mathrm{tot}}$ is the related total percentage.
In addition, we report the minimal and maximal eigenvalue of the preconditioned matrices.
The same information is reported for the matrices $iP_n^{-1}(\a)\mathrm{Im}(A_n(\a,\b))$, but
with respect to a cluster at $0$. \newline
\begin{table}
\caption{Number of PHSS/IPHSS outer iterations and average per outer step for PCG and PGMRES inner iterations
(total number of inner iterations in brackets).}
\label{tab:IT-ES1M1}
\begin{center} \footnotesize
\begin{tabular}{|c|c|cc|c|cc|}
 \hline
 \multicolumn{7}{|c|}{$\a(x,y)=\exp(x+y)$, $\b(x,y)=[x\ y]^T$}  \\
 \hline
 $n$  & PHSS & PCG & PGMRES & IPHSS & PCG & PGMRES \\
 \hline
 81    & 5 & 1.6 (8) & 2.4 (12) & 5 & 1 (5) & 1 (5)\\
 361   & 5 & 1.6 (8) & 2.8 (14) & 5 & 1 (5) & 1 (5)\\
 1521  & 5 & 1.6 (8) & 3   (15) & 5 & 1 (5) & 2 (10)\\
 6241  & 5 & 1.6 (8) & 3.2 (16) & 5 & 1 (5) & 2 (10)\\
 25281 & 5 & 1.6 (8) & 3.6 (18) & 5 & 1 (5) & 2 (10)\\
 \hline
\end{tabular}
\end{center}
\end{table}
\begin{table}
\caption{Outliers analysis.}
\label{tab:OUTLIER-ES1M1}
\begin{center} \footnotesize
\begin{tabular}{|c|ccc|cc|ccc|cc|}
 \hline
 \multicolumn{11}{|c|}{$\a(x,y)=\exp(x+y)$, $\b(x,y)=[x\ y]^T$}  \\
 \hline
 $n$ & \multicolumn{5}{|c|}{$P_n^{-1}(\a)\mathrm{Re}(A_n(\a,\b))$} & \multicolumn{5}{|c|}{$P_n^{-1}(\a)\mathrm{Im}(A_n(\a,\b))$}   \\
     & $m_-$ & $m_+$ & $p_{\mathrm{tot}}$ & $\lambda_{\min}$ & $\lambda_{\max}$  & $m_-$ & $m_+$ & $p_{\mathrm{tot}}$ & $\lambda_{\min}$ & $\lambda_{\max}$ \\
 \hline
81   &0 &0 &0\%    & 9.99e-01 & 1.04e+00 &  0 &0 &0\%    & -2.68e-02 & 2.68e-02\\
     &0 &3 &3\%    &          &          &  4 &4 &9\%    & & \\
\hline
361  &0 &0 &0\%    & 9.99e-01 & 1.04e+00 &  0 &0 &0\%    & -2.87e-02 & 2.87e-02\\
     &0 &4 &1\%    &          &          &  7 &7 &3.8\%  & & \\
\hline
1521 &0 &0 &0\%    & 9.99e-01 & 1.044e+0 &  0 &0 &0\%    & -2.93e-02 & 2.93e-02\\
     &0 &4 &0.26\% &          &          &  9 &9 &1.18\% & & \\
\hline
\end{tabular}
\end{center}
\end{table}
Despite the lack of the corresponding theoretical results, we want to test the PHSS
convergence behavior also in the case in which the regularity assumption on $\a(x,y)$
in Theorems \ref{teo:clu+se_ReA} and \ref{teo:clu+sb_ImA} are not satisfied. The analysis is motivated by
favorable known numerical results in  the case of FD approximations (see, for instance,
\cite{ST-ETNA-2003,ST-SIMAX-2003,BGST-NM-2005}) or FE approximation with only the diffusion term
\cite{ST-NA-2001}. More precisely, we consider as template the $\mathcal{C}^1$ function
$\a(x,y)=\a_2(x,y)=e^{x+|y-1/2|^{3/2}}$,
the $\mathcal{C}^0$ function $\a(x,y)=\a_3(x,y)=e^{x+|y-1/2|}$, and the piecewise constant function $\a(x,y)=\a_4(x,y)=1$ if $y<1/2$, $10$ otherwise.\newline
The number of required PHSS outer iterations is listed in Table \ref{tab:IT-ES23M1}, together
with the averages per outer step for PCG and PGMRES inner iterations (the total is in brackets).
Notice that in the case of the $\mathcal{C}^1$ or $\mathcal{C}^0$ function the outer iteration number
does not depend on the coefficient matrix dimension $n=n(h)$. The same seems to be true
with respect to the PCG inner iterations, while the PGMRES inner iterations show some influence on $n$. More precisely, this influence
lies in a higher number of PGMRES inner iterations in the first few steps of the PHSS outer iterations.
Moreover, the considered  variable tolerance stopping criterion for the inner iterations  devised in the IPHSS
method is just able to reduce this sensitivity. \newline
These remarks are in perfect agreement with the outliers analysis of the matrices $P_n^{-1}(\a)\mathrm{Re}(A_n(\a,\b))$
and $P_n^{-1}(\a)\mathrm{Im}(A_n(\a,\b))$, with respect to a cluster at $1$ and at $0$, respectively,
reported in Table \ref{tab:OUTLIER-ES23M1}, with the same notations as before. \newline
Separate mention has to be made to the case of the piecewise continuous function $\a(x,y)=\a_4(x,y)$. In fact, even if
$\{P_n^{-1}(\a)\mathrm{Im}(A_n(\a,\b))\}$ could be supposed to be strongly clustered at $0$, it is evident that
$\{P_n^{-1}(\a)\mathrm{Re}(A_n(\a,\b))\}$ is clustered at $1$, but not in a strong way: the number of the outliers grows for increasing $n$,
though their percentage is decreasing, in accordance with the notion of weak clustering
(see Table \ref{tab:OUTLIER-ES4M1}). Indeed, the number of outer iterations grows for increasing $n$ as shown in Table \ref{tab:IT-ES4M1} and
a deeper insight allows to observe that the major difficulty is, as expected, in the PCG inner iterations.
The same behavior is observed also when varying the parameter $\alpha$, in order to find the optimal setting (see Table \ref{tab:ITALPHA-ES4M1}).
\par
Lastly, we want to test our proposal in the case of other structured and unstructured meshes
generated by triangle \cite{Triangle} with a progressive refinement procedure. The first meshes in the considered
mesh sequences are reported in Figures \ref{fig:mesh_M2}-\ref{fig:mesh_M4}.\newline
Tables \ref{tab:IT-ES123M2}-\ref{tab:IT-ES123M4} report the number of required PHSS/IPHSS iterations in the case of the
previous template functions. Negligible differences in the PHSS/IPHSS outer iterations are observed for increasing
dimensions $n$. Again, some dependency on $n$ is observed with respect to the PGMRES inner iterations, due to
an higher number of PGMRES inner iterations required in the first few steps of the PHSS outer iterations in relation to
a more severe ill-conditioning.
Clearly, a more sophisticated stopping criterion may probably reduce this sensitivity. \newline
\begin{table}
\caption{Number of PHSS/IPHSS outer iterations and average per outer step for PCG and PGMRES inner iterations
(total number of inner iterations in brackets).}\label{tab:IT-ES23M1}
\begin{center} \footnotesize
\begin{tabular}{|c|c|ll|c|ll|}
 \hline
 \multicolumn{7}{|c|}{$\a_2(x,y)$, $\b(x,y)=[x \ y]^T$} \\
 \hline
 $n$  & PHSS & PCG & PGMRES & IPHSS & PCG & PGMRES  \\
 \hline
 81    & 6 &2.2 (13) &2.8 (17) & 6 & 1 (6) & 1 (6)  \\
 361   & 6 &2.2 (13) &3.2 (19) & 6 & 1 (6) & 2 (12) \\
 1521  & 6 &2.2 (13) &3.5 (21) & 6 & 1 (6) & 2 (12) \\
 6241  & 6 &2.2 (13) &4   (24) & 6 & 1 (6) & 2 (12) \\
 25281 & 6 &2.2 (13) &4.2 (25) & 6 & 1 (6) & 3 (18) \\
 \hline
 \multicolumn{7}{|c|}{$\a_3(x,y)$, $\b(x,y)=[x \ y]^T$} \\
 \hline
 $n$  & PHSS & PCG & PGMRES & IPHSS & PCG & PGMRES  \\
 \hline
 81    & 7 &1.9 (13) &2.6 (18) & 7 & 1 (7) & 1 (7)      \\
 361   & 7 &2.2 (15) &3   (21) & 7 & 1 (7) & 1.7 (12)   \\
 1521  & 7 &2.2 (15) &3.5 (24) & 7 & 1.1 (8) & 2 (14)   \\
 6241  & 7 &2.3 (16) &3.6 (25) & 7 & 1.1 (8) & 2 (14)   \\
 25281 & 7 &2.3 (16) &4   (28) & 7 & 1.1 (8) & 2.1 (15) \\
 \hline
\end{tabular}
\end{center}
\end{table}
\begin{table}
\caption{Outliers analysis.}
\label{tab:OUTLIER-ES23M1}
\begin{center} \footnotesize
\begin{tabular}{|c|ccc|cc|ccc|cc|}
\hline
\multicolumn{11}{|c|}{$\a_2(x,y)$, $\b(x,y)=[x\ y]^T$}  \\
\hline
$n$ & \multicolumn{5}{|c|}{$P_n^{-1}(\a)\mathrm{Re}(A_n(\a,\b))$}
& \multicolumn{5}{|c|}{$P_n^{-1}(\a)\mathrm{Im}(A_n(\a,\b))$} \\
\hline
81   &0 &1 &1.2\%  &9.97e-01 &1.12e+00 &0 &0 &0\%  &-4.32e-02 &4.32e-02 \\
     &0 &9 &11\%   &          &         &7 &7 &17\% && \\
\hline
361  &0 &1  &0.27\% &9.99e-01 &1.12e+00 &0  &0 &0\%  &-4.68e-02 &-4.68e-02 \\
     &0 &11 &3\%    &         &         &15 &15 &8.3\%&& \\
\hline
1521 &0 &1  &6\%    &9.99e-01 &1.12e+00 &0  &0  &0\% &-4.78e-02& 4.78e-02\\
     &0 &12 &0.79\% &         &         &21 &21 &2.8\% && \\
\hline
\multicolumn{11}{|c|}{$\a_3(x,y)$, $\b(x,y)=[x\ y]^T$}  \\
\hline
$n$ & \multicolumn{5}{|c|}{$P_n^{-1}(\a)\mathrm{Re}(A_n(\a,\b))$}
& \multicolumn{5}{|c|}{$P_n^{-1}(\a)\mathrm{Im}(A_n(\a,\b))$} \\
\hline
81   &0 &1 &1.2\%   & 9.95e-01  &  1.16e+000 & 0 &0& 0\%    & -3.97e-02 & 3.97e-02 \\
     &0 &9 &11 \%   & & & 6 &6 &14\%   & & \\
     \hline
361  &0 &1  &0.28\% & 9.97e-01  &   1.17e+00 &0  &0  &0\%   & -4.31e-02 & 4.31e-02\\
     &0 &11 &3\%    & & &13 &13 &7\%   & & \\
     \hline
1521 &0 &1  &0.07\% & 9.98e-01  &  1.18e+00 &0  &0  &0\%   & -4.40e-02 & 4.40e-02\\
     &0 &14 &0.92\% & & &18 &18 &2.4\% & & \\
\hline
%
%
\end{tabular}
\end{center}
\end{table}
\begin{table}
\caption{Outliers analysis.}
\label{tab:OUTLIER-ES4M1}
\begin{center} \footnotesize
\begin{tabular}{|c|ccc|cc|ccc|cc|}
\hline
\multicolumn{11}{|c|}{$\a_4(x,y)$, $\b(x,y)=[x\ y]^T$}  \\
\hline

$n$ & \multicolumn{5}{|c|}{$P_n^{-1}(\a)\mathrm{Re}(A_n(\a,\b))$}
& \multicolumn{5}{|c|}{$P_n^{-1}(\a)\mathrm{Im}(A_n(\a,\b))$} \\
\hline
81   &9 &7 &19\%     & 5.84e-01 &   2.09e+00  &0 &0 &0\%    & -2.23e-02 & 2.23e-02 \\
     &9 &9 &22\%     & &  &1 &1 &2.5\%  & & \\
\hline
361  &19 &17 &9.8\%  & 4.20e-01  &  2.97e+00  &0 &0 &0\%    & -2.99e-02 & 2.99e-02\\
     &19 &20 &10.8\% & &  &3 &3 &1.6\%  & & \\
\hline
1521 &39 &37 &5\%    & 2.78e-01  &  4.53e+000  &0 &0 &0\%    & -3.34e-02 & 3.34e-02\\
     &39 &40 &5.2\%  & &  &6 &6 &0.79\% & &  \\
\hline
\end{tabular}
\end{center}
\end{table}
\begin{table}
\caption{Number of PHSS/IPHSS outer iterations and average per outer step for PCG and PGMRES inner iterations
(total number of inner iterations in brackets).}\label{tab:IT-ES4M1} 
\begin{center} \footnotesize
\begin{tabular}{|c|c|ll|c|ll|}
 \hline
 \multicolumn{7}{|c|}{$\a_4(x,y)$, $\b(x,y)=[x \ y]^T$} \\
 \hline
 $n$  & PHSS & PCG & PGMRES & IPHSS & PCG & PGMRES  \\
 \hline
 81    & 13 & 3.5 (45)  & 2.2 (29) & 13 & 1.9 (25)  & 1 (13)\\
 361   & 20 & 3.9 (78)  & 2.3 (47) & 20 & 2.1 (41)  & 1 (20)\\
 1521  & 31 & 4.3 (134) & 2.4 (74) & 32 & 2.2 (70)  & 1.5 (47)  \\
\hline
\end{tabular}
\end{center}
\end{table}
\begin{table}
\caption{Number of PHSS/IPHSS outer iterations and average per outer step for PCG and PGMRES inner iterations
(total number of inner iterations in brackets) in the case of optimal $\alpha^*$ values.}\label{tab:ITALPHA-ES4M1} 
\begin{center} \footnotesize
\begin{tabular}{|c|c|c|ll|c|c|ll|}
 \hline
 \multicolumn{9}{|c|}{$\a_4(x,y)$, $\b(x,y)=[x \ y]^T$} \\
 \hline
 $n$  & PHSS & $\alpha^*$ & PCG & PGMRES  & IPHSS & $\alpha^*$ & PCG & PGMRES \\
 \hline
 81    & 12 & (1.068,1.12)   & 3.8 (45)  & 2.2 (27)  
       & 12 & (1.066,1.114)  & 2   (24)  & 1 (12)  \\ 
 361   & 19 & (1.04,1.1656)  & 4.1 (77)  & 2.4 (45)  
       & 18 & (1.088,1.1024) & 2.1 (37)  & 1 (18)  \\ 
 1521  & 29 & (1.02,1.0736)  & 4.5 (131) & 2.4 (70)  
       & 30 & (1.0526, 1.16) & 2.1 (64)  & 1.4 (43)  \\ 
\hline
\end{tabular}
\end{center}
\end{table}
\begin{table}
\caption{Number of PHSS/IPHSS outer iterations and average per outer step for PCG and PGMRES inner iterations
(total number of inner iterations in brackets) - meshes in Fig. \ref{fig:mesh_M2}.}\label{tab:IT-ES123M2}
\begin{center} \footnotesize
\begin{tabular}{|l|l|ll|l|ll|}
 \hline
  \multicolumn{7}{|c|}{$\a_1(x,y)$, $\b(x,y)=[x \ y]^T$ }  \\
 \hline
 $n$  & PHSS & PCG & PGMRES & IPHSS & PCG & PGMRES  \\ \hline
 41    &5 &2.2 (11) &2.2 (11) & 5 & 1 (5) & 1 (5)\\
 181   &5 &2.2 (11) &2.6 (13) & 5 & 1 (5) & 1 (5) \\
 761   &5 &2.2 (11) &3 (15)   & 5 & 1 (5) & 1.2 (6) \\
 3121  &5 &2.2 (11) &3 (15)   & 5 & 1 (5) & 2 (10) \\
 12641 &5 &2.2 (11) &3.4 (17) & 5 & 1 (5) & 2 (10) \\
 50881 &5 &2.2 (11) &3.8 (19) & 5 & 1 (5) & 2 (10) \\
 \hline
 \multicolumn{7}{|c|}{$\a_2(x,y)$, $\b(x,y)=[x \ y]^T$ }  \\
 \hline
 $n$  & PHSS & PCG & PGMRES & IPHSS & PCG & PGMRES  \\ \hline
 41    & 6 &2.2 (13) &2.7 (16) &  6 & 1 (6) & 1 (6)  \\
 181   & 6 &2.2 (13) &3 (18)   &  6 & 1 (6) & 1 (6) \\
 761   & 6 &2.2 (13) &3.3 (20) &  6 & 1 (6) & 2 (12) \\
 3121  & 6 &2.2 (13) &3.7 (22) &  6 & 1 (6) & 2 (12) \\
 12641 & 6 &2.2 (13) &4 (24)   &  6 & 1 (6) & 2.2 (13) \\
 50881 & 6 &2.2 (13) &4.3 (26) &  6 & 1 (6) & 3 (18) \\
 \hline
 \multicolumn{7}{|c|}{$\a_3(x,y)$, $\b(x,y)=[x \ y]^T$ }  \\
 \hline
 $n$  & PHSS & PCG & PGMRES & IPHSS & PCG & PGMRES\\ \hline
 41    &7 &2.2 (15) &2.5 (17) &7 & 1 (7) & 1 (7)\\
 181   &7 &2.2 (15) &2.8 (20) &7 & 1 (7) & 1 (7)\\
 761   &7 &2.2 (15) &3.1 (22) &7 & 1.1 (8) & 2 (14)\\
 3121  &7 &2.4 (17) &3.6 (25) &7 & 1.1 (8) & 2 (14)\\
 12641 &7 &2.4 (17) &3.8 (27) &7 & 1.1 (8) & 2 (14)\\
 50881 &7 &2.4 (17) &3.9 (31) &8 & 1.1 (9) & 2.2 (18)\\
\hline
\end{tabular}
\end{center}
\end{table}
\begin{table}
\caption{Number of PHSS/IPHSS outer iterations and average per outer step for PCG and PGMRES inner iterations
(total number of inner iterations in brackets)  - meshes in Fig. \ref{fig:mesh_M3}.}\label{tab:IT-ES123M3}
\begin{center} \footnotesize
\begin{tabular}{|l|l|ll|l|ll|}
 \hline
\multicolumn{7}{|c|}{$\a_1(x,y)$, $\b(x,y)=[x \ y]^T$}  \\
 \hline
 $n$  & PHSS & PCG & PGMRES & IPHSS & PCG & PGMRES  \\ \hline
 25    & 5 &2.2 (11) &2.2(11)  & 5 & 1 (5) & 1 (5)  \\
 113   & 5 &2.2 (11) &2.4 (12) & 5 & 1 (5) & 1 (5)  \\
 481   & 5 &2.2 (11) &2.8 (14) & 5 & 1 (5) & 1 (5)  \\
 1985  & 5 &2.2 (11) &3 (15)   & 5 & 1 (5) & 2 (10) \\
 8065  & 5 &2.2 (11) &3.4 (17) & 5 & 1 (5) & 2 (10) \\
 32513 & 5 &2.2 (11) &3.6 (18) & 5 & 1 (5) & 2 (10) \\
\hline
\multicolumn{7}{|c|}{$\a_2(x,y)$, $\b(x,y)=[x \ y]^T$}  \\
\hline
 $n$  & PHSS & PCG & PGMRES & IPHSS & PCG & PGMRES  \\ \hline
 25    &6 & 2.2 (13) & 2.5 (15) & 6 & 1 (6) & 1 (6) \\
 113   &6 & 2.2 (13) & 3 (18)   & 6 & 1 (6) & 1 (6) \\
 481   &6 & 2.2 (13) & 3.2 (19) & 6 & 1 (6) & 2 (12) \\
 1985  &6 & 2.2 (13) & 3.7 (22) & 6 & 1 (6) & 2 (12) \\
 8065  &6 & 2.2 (13) &4 (24)    & 6 & 1 (6) & 2 (12)\\
 32513 &6 & 2.2 (13) &4.3 (26)  & 6 & 1 (6) & 3 (18)\\
\hline
\multicolumn{7}{|c|}{$\a_3(x,y)$, $\b(x,y)=[x \ y]^T$}  \\
\hline

 $n$  & PHSS & PCG & PGMRES & IPHSS & PCG & PGMRES  \\ \hline
 25    &6 &2.2 (13) &2.5 (15) & 6 & 1 (6)   & 1 (6) \\
 113   &7 &2 (14)   &2.7 (19) & 7 & 1 (7)   & 1 (7) \\
 481   &7 &2.1 (15) &3 (21)   & 7 & 1.1 (8) & 1.8 (13) \\
 1985  &7 &2.3 (16) &3.4 (24) & 7 & 1.1 (8) & 2 (14) \\
 8065  &7 &2.4 (17) &3.8 (27) & 7 & 1.1 (8) & 2 (14)\\
 32513 &8 &2.1 (17) &3.8 (30) & 8 & 1.1 (9) & 2.1 (17)\\
 \hline
\end{tabular}
\end{center}
\end{table}
\begin{table}
\caption{Number of PHSS/IPHSS outer iterations and average per outer step for PCG and PGMRES inner iterations
(total number of inner iterations in brackets)  - meshes in Fig. \ref{fig:mesh_M4}.}\label{tab:IT-ES123M4}
\begin{center} \footnotesize
\begin{tabular}{|l|l|ll|l|ll|}
\hline
\multicolumn{7}{|c|}{$\a_1(x,y)$, $\b(x,y)=[x \ y]^T$ } \\
\hline
 $n$  & PHSS & PCG & PGMRES & IPHSS & PCG & PGMRES  \\ \hline
 55    &5 &2.2 (11) &2.2 (11) & 5 & 1 (5) &  1 (5) \\
 142   &5 &2.2 (11) &2.6 (13) & 5 & 1 (5) &  1 (5) \\
 725   &5 &2.2 (11) &3 (15)   & 5 & 1 (5) &  1 (5)\\
 1538  &5 &2.2 (11) &3 (15)   & 5 & 1.2 (6) & 1.2 (6)\\
 7510  &5 &2.2 (11) &3 (17)   & 5 & 1.2 (6) & 1.8 (9)\\
 15690 &5 &2.2 (12) &3.4 (18) & 6 & 1.2 (7) & 2 (12)\\
 \hline
 \multicolumn{7}{|c|}{$\a_2(x,y)$, $\b(x,y)=[x \ y]^T$} \\
\hline
 $n$  & PHSS & PCG & PGMRES & IPHSS & PCG & PGMRES  \\ \hline
 55    &6 &2.2 (13) &2.7 (16) & 6 & 1 (6)   & 1 (6) \\
 142   &6 &2.2 (13) &3 (18)   & 6 & 1 (6)   & 1 (6) \\
 725   &6 &2.2 (13) &3.3 (20) & 6 & 1 (6)   & 2 (12) \\
 1538  &6 &2.2 (13) &3.5 (21) & 6 & 1.2 (7) & 2 (12)\\
 7510  &7 &2 (14)   &3.7 (26) & 7 & 1.1 (8) & 2 (14)\\
 15690 &7 &2 (14)   &3.7 (26) & 7 & 1.1 (8) & 2 (14) \\
\hline
\multicolumn{7}{|c|}{$\a_3(x,y)$, $\b(x,y)=[x \ y]^T$ } \\
\hline
 $n$  & PHSS & PCG & PGMRES & IPHSS & PCG & PGMRES  \\ \hline
 55    &7 &1.8 (13) &2.4 (17) & 7 & 1 (7)   & 1 (7) \\
 142   &7 &2.2 (15) &2.8 (20) & 7 & 1 (7)   & 1 (7) \\
 725   &7 &2.6 (18) &3.1 (22) & 7 & 1.1 (8) & 2 (14) \\
 1538  &7 &2.6 (18) &3.4 (24) & 7 & 1.1 (8) & 2 (14) \\
 7510  &7 &2.6 (18) &3.7 (26) & 7 & 1.1 (8) & 2 (14) \\
 15690 &8 &2.4 (19) &3.6 (29) & 8 & 1.1 (9) & 2 (16)\\
\hline
\end{tabular}
\end{center}
\end{table}
\begin{figure}
\centering
\vskip -0.5cm
\epsfig{file=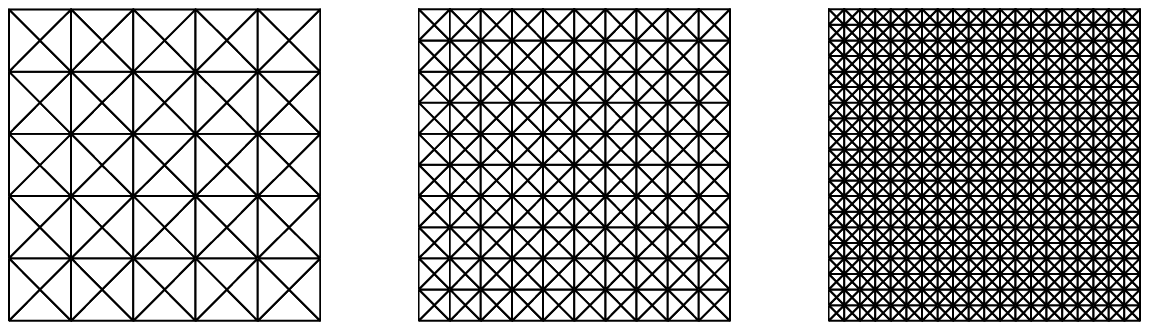,width=\textwidth}
\vskip -3.5cm
\caption{Structured meshes.} 
\label{fig:mesh_M2}
\end{figure}
\clearpage
\begin{figure}
\centering
\vskip -0.5cm
\epsfig{file=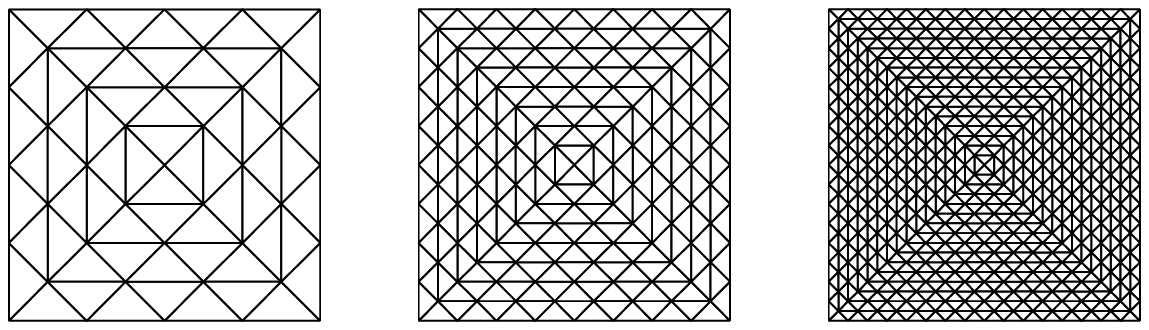,width=\textwidth}
\vskip -3.5cm
\caption{Structured meshes.} 
\label{fig:mesh_M3}
\end{figure}
\begin{figure}
\centering
\vskip -0.5cm
\epsfig{file=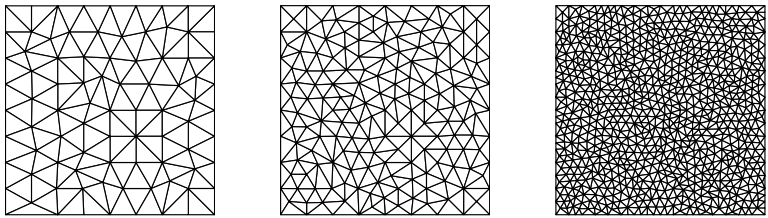,width=\textwidth}
\vskip -3.5cm
\caption{Unstructured meshes.} 
\label{fig:mesh_M4}
\end{figure}
\section{Complexity issues and perspectives} \label{sez:conclusions}
Lastly, we report some remarks about the computational costs of the proposed iterative
procedure by referring to the optimality definition below.
\ \newline
\begin{definition}\emph{\cite{AN-1993}} \label{def:opt}
Let $\{A_m \mathbf{x}_m = \mathbf{b}_m\}$ be a given sequence of
linear systems of increasing dimensions. An iterative method is
  \emph{optimal} if
  \begin{enumerate}
    \item[{\rm 1.}] the arithmetic cost of each iteration is at most
      proportional to the complexity of a matrix vector product with
      matrix $A_m$,
    \item[{\rm 2.}] the number of iterations for reaching the solution within a
      fixed accuracy can be bounded from above by a constant
      independent of $m$.
  \end{enumerate}
\end{definition}
\ \newline \noindent
In other words, the problem of solving a linear system with coefficient matrix $A_m$
is asymptotically of the same cost as the direct problem of multiplying $A_m$ by a vector.
\par
Since we are considering  the preconditioning matrix sequence $\{P_n(\a)\}$  defined as
$P_n(\a)=D_n^{{1}/{2}}(\a) A_n(1,0) D_n^{{1}/{2}}(\a)$, where $D_n(\a)=\mathrm{diag}(A_n(\a,0))\mathrm{diag}\!^{-1}(A_n(1,0))$,
the solution of the linear system in (\ref{eq:modello_discreto}) with matrix $A_n(\a,\b)$ is reduced to computations
involving diagonals and the matrix $A_n(1,0)$. \newline
As well known, whenever the domain is partitioned
by considering a uniform structured mesh this latter task can be efficiently performed by means of fast Poisson solvers,
among which we can list those based on the cyclic reduction idea (see e.g.
\cite{BDGG-SINUM-1971,D-SIAMRev-1970,S-SIAMRev-1977}) and several specialized multigrid methods
(see e.g. \cite{H-Springer-1985,S-NM-2002}).
Thus, in such a setting, and under the regularity assumptions 
(\ref{eq:ipotesi_coefficienti}), the optimality of the PHSS method is
theoretically proved: the PHSS iterations number for reaching the solution within a fixed accuracy can be bounded from above by a constant
independent of the dimension $n=n(h)$ and the arithmetic cost of each iteration is at most proportional to the complexity of a matrix vector product
with matrix $A_n(\a,\b)$. \par
Finally, we want to stress that the PHSS numerical performances do not get worse in the case of unstructured meshes.
In such cases, again, our proposal makes only use of matrix vector products (for sparse or even diagonal matrices)
and of a solver for the related diffusion equation with constant coefficient. To this end, the main effort in devising efficient
algorithms must be devoted only to this simpler problem.

%
\end{document}